\documentclass[]{amsart}
\usepackage{amsmath}
\usepackage[english]{babel}
\usepackage[utf8]{inputenc}
\usepackage{amsfonts}
\usepackage{amsthm}
\usepackage{color}
\numberwithin{equation}{section}

\newtheorem{thm}{Theorem}[section]
\newtheorem{defn}{Definition}[section]

\newtheorem{prop}[thm]{Proposition}
\newtheorem{cor}[thm]{Corollary}
\newtheorem{lem}[thm]{Lemma}

\theoremstyle{remark}
\newtheorem{rmk}[thm]{Remark}
\theoremstyle{definition}

\DeclareMathOperator{\Cov}{Cov}
\DeclareMathOperator{\BF}{\mathcal{BF}}

\DeclareMathOperator{\E}{\mathbb{E}}
\DeclareMathOperator{\m}{\mathbf{m}}
\DeclareMathOperator{\N}{\mathbb{N}}
\DeclareMathOperator{\Z}{\mathbb{Z}}
\DeclareMathOperator{\R}{\mathbb{R}}
\DeclareMathOperator{\cF}{\mathcal{F}}
\DeclareMathOperator{\cG}{\mathcal{G}}
\DeclareMathOperator{\sign}{sign}

\DeclareMathOperator{\cL}{\mathcal{L}}

\DeclareMathOperator{\fe}{\mathfrak{e}}

\DeclareMathOperator{\fd}{\mathfrak{d}}
\DeclareMathOperator{\bP}{\mathbb{P}}

\DeclareMathOperator{\mm}{\mathbf{m}}

\newcommand{\pd}[2]{\frac{\partial #1}{\partial #2}}
\newcommand{\pdsup}[3]{\frac{\partial^{#3} #1}{\partial #2^{#3}}}

\newcommand{\der}[2]{\frac{d #1}{d #2}}
\newcommand{\dersup}[3]{\frac{d^{#3} #1}{d #2^{#3}}}

\newcommand{\Norm}[2]{\left\Vert #1 \right\Vert_{#2}}

\title{Non-Local Solvable Birth-Death Processes}
\author{Giacomo Ascione$^\ast$}
\address{$^\ast$ Dipartimento di Matematica e Applicazioni ``Renato Caccioppoli'', Università degli Studi di Napoli Federico II, 80126 Napoli, Italy}
\author{Nikolai Leonenko$^\dagger$}
\address{$^\dagger$ School of Mathematics, Cardiff University, Cardiff CF24 4AG, UK}
\author{Enrica Pirozzi$^\ast$}
\email{giacomo.ascione@unina.it \\
	leonenkon@cardiff.ac.uk \\
	enrica.pirozzi@unina.it}

\begin{document}
\maketitle
\begin{abstract}
In this paper we study strong solutions of some non-local difference-differential equations linked to a class of birth-death processes arising as discrete approximations of Pearson diffusions by means of a spectral decomposition in terms of orthogonal polynomials and eigenfunctions of some non-local derivatives. Moreover, we give a stochastic representation of such solutions in terms of time-changed birth-death processes and study their invariant and their limit distribution. Finally, we describe the correlation structure of the aforementioned time-changed birth-death processes. 
\end{abstract}
\section{Introduction}
Birth-death processes constitute a fundamental class of continuous-time Markov chain that are widely used in applications such as, for instance, evolutionary dynamics \cite{nowak2006evolutionary} and queueing theory \cite{kleinrock1975queueing}. In particular, one can achieve a complete characterization of birth-death processes by families of classical orthogonal polynomials of discrete variable. This theory, linked to the solution of the Stieltjes moment problem, has been widely studied by Karlin and McGregor in their seminal papers \cite{karlin1957classification,karlin1957differential}.\\
As birth-death processes are linked to difference-differential equations, \textit{fractionalization} of such processes can be used to study the solutions of fractional difference-differential equations. Indeed, with this idea in mind, a fractional version of the Poisson process has been introduced in \cite{beghin2009fractional,beghin2010poisson} and fractional versions of some birth-death processes, for instance, in \cite{orsingher2010fractional2,orsingher2011fractional,orsingher2010fractional}. \\
In the case of Pearson diffusions, one can use a spectral approach to study strong solutions of fractional backward and forward Kolmogorov equations and, at the same time, define the fractional Pearson diffusions by means of a time-change via an inverse stable subordinator (see, for instance, \cite{leonenko2013correlation,leonenko2013fractional,leonenko2017heavy}). The same approach has been used to study the case of fractional immigration-death processes in \cite{ascione2019fractional}. Let us also stress out that this approach can be used to study a fractional $M/M/\infty$ queue or a fractional $M/M/1$ queue with acceleration of service (for some models of fractional queues, we refer to \cite{ascione2018fractional,ascione2020fractional,cahoy2015transient}).\\
However, one could consider a time-change with a different inverse subordinator. In such case, in place of the fractional derivative in time, one obtains a more general non-local operator. Such kind of operators have been introduced in \cite{kochubei2011general} for the class of complete Bernstein functions and extended in \cite{toaldo2015convolution} for any Bernstein function. A first step towards the theory of general time-changed Pearson diffusions has been achieved in \cite{gajda2015time}.\\
In this work, we describe a general theory for non-local solvable birth-death processes in terms of orthogonal polynomials, where such processes are defined by means of a time-change with a general inverse subordinator. In particular, we focus on the strong solutions of general non-local backward and forward Kolmogorov equations associated to such processes and on the stochastic representation of such solutions. In particular, the paper is structured as follows:
\begin{itemize}
	\item In Section \ref{Sec2} we introduce the theory of solvable birth-death processes as discrete approximations of Pearson diffusions and we state the main hypotheses we need on the starting birth-death process;
	\item In Section \ref{Sec3} we give some preliminaries on inverse subordinators and non-local time derivatives. In particular we focus on the eigenvalue equation for such derivatives and on some upper bounds for the eigenfunctions. Let us stress out that some properties of such eigenfunctions are expressed in \cite{kochubei2011general,kochubei2019growth} in the complete Bernstein case and in \cite{meerschaert2019relaxation} in the general case. Moreover, a series expansion in terms of convolutions of potential densities in the special Bernstein case is obtained in \cite{ascione2020generalized};
	\item In Section \ref{Sec4} we focus on the spectral decomposition of the strong solutions of non-local forward and backward Kolmogorov equations in terms of orthogonal polynomials of discrete variable and eigenfunctions of the non-local time derivatives;
	\item In Section \ref{Sec5} we introduce the time-changed birth-death processes and we study the stochastic representation of the aforementioned strong solutions in terms of such processes. In particular we obtain that the time-changed process still admits the same invariant measure that is also the limit measure for any starting distribution;
	\item Finally, in Section \ref{Sec6} we study the correlation structure of the time-changed birth-death processes in terms of the potential measure of the involved subordinator and the eigenfunctions of the non-local time derivatives. In particular, the non-stationarity of the process is evident in the expression of the covariance, thus, to give some information on the memory of the process, we have to refer to a non-stationary extension of the definition of long-range and short-range dependence suggested by the necessary conditions given in \cite[Lemma $2.1$ and $2.2$]{beran2016long}.
\end{itemize}
	\section{Solvable Birth-Death Processes}\label{Sec2}
Let us fix a filtered space $(\Omega, \cF, \{\cF\}_{t \in \R^+} \bP)$ and consider a Birth-Death process $\{N(t),t \ge 0\}$ on it. Let us denote by $E \subseteq \Z$ its state space, that will be finite or at most countable. In particular we can always suppose that $E \subseteq \N_0$ and is a segment, i.e. for any $n_1,n_2 \in E$ and $n \in \N_0$ such that $n_1 \le n\le n_2$ it holds $n \in E$, with $\min E=0$. Let us recall that the generator $\cG$ of a Birth-Death process can be always expressed as
\begin{equation*}
\cG f(x)=(b(x)-d(x))\nabla^+ f(x)+d(x)\Delta f(x), \qquad x \in E,
\end{equation*}
where $d(x)$ are the death rates, $b(x)$ are the birth rates (recalling that $d,b$ must be non-negative in $E$), $\nabla^{\pm}$ are the first order forward and bacwkard finite differences defined as
\begin{align*}
\nabla^+ f(x)=f(x+1)-f(x) && \nabla^- f(x)=f(x)-f(x-1)
\end{align*}
and $\Delta$ is the second order central finite difference
\begin{equation*}
\Delta f(x)=f(x+1)-2f(x)+f(x-1)=\nabla^-\nabla^+f(x).
\end{equation*}
Here we want to introduce some birth-death version of Pearson diffusions (see, for instance, \cite{leonenko2013fractional}). To do this, we refer to the theory of birth-death polynomials, whose main papers are \cite{karlin1957classification,karlin1957differential}. 
\begin{defn}
We say the process $N(t)$ is solvable if
\begin{itemize}
	\item $N(t)$ is irreducible and recurrent;
	\item the spectrum of $\cG$ is purely discrete with non-positive eigenvalues $(\lambda_n)_{n \in E}$ such that $\lambda_0=0$ and $\lambda_n<0$ for any $n \ge 1$;
	\item  its eigenfunctions $(P_n)_{n \in E}$ are classical orthogonal polynomials of discrete variable with orthogonality measure $\m$ which is the invariant and stationary measure of $N(t)$;
	\item the function $m(x)=\m(\{x\})$ solves the following discrete Pearson equation:
	\begin{equation}\label{dPe}
	\nabla^+(d(\cdot)m(\cdot))(x)=(b(x)-d(x))m(x)
	\end{equation}
	\item $d(\cdot)$ is a polynomial of degree at most $2$ and $b(\cdot)-d(\cdot)$ is a polynomial of degree at most $1$.
\end{itemize}
\end{defn}
Concerning solvable birth-death processes, they arise as lattice approximations of Pearson diffusions.
In particular, one has in such case
\begin{equation*}
\lambda_n=n \nabla^+(b(\cdot)-d(\cdot))(x)+\frac{1}{2}n(n-1)\Delta d(x).
\end{equation*}
Concerning classical orthogonal polynomials of discrete variable, we mainly refer to \cite{nikiforov1991classical,schoutens2012stochastic}. Their orthogonality relation is expressed as
\begin{equation*}
\sum_{x \in E}P_n(x)P_m(x)m(x)=\fd_n^2 \delta_{n,m}, \qquad \ n,m \in E
\end{equation*}
 where $\delta_{n,m}$ is Kronecker delta symbol. In particular one obtains that $\Norm{P_n}{\ell^2(\m)}=\fd_n$ and then we can introduce the normalized polynomials as $Q_n(x)=\frac{P_n(x)}{\fd_n}$, such that
 \begin{equation*}
 \sum_{x \in E}Q_n(x)Q_m(x)m(x)=\delta_{n,m}, \qquad \ n,m \in E.
 \end{equation*}
 On the other hand, the function $\widetilde{m}(n)=\frac{1}{\fd^2_n}$ defines a measure on $E$. Thus, by proceeding with a Gram-Schmidt orthogonalization procedure on the monomials $(1,x,x^2,\cdot)$, we can define a family of orthogonal polynomials $\widetilde{P}_n(x)$ that satisfies the following orthogonality condition
 \begin{equation*}
 \sum_{x \in E}\widetilde{P}_n(x)\widetilde{P}_m(x)\widetilde{m}(x)=\frac{1}{m(n)} \delta_{n,m}, \qquad \ n,m \in E.
 \end{equation*}
  The family of polynomials $(\widetilde{P}_n)_{n \in E}$ are called the dual family of $(P_n)_{n \in E}$, see \cite{nikiforov1991classical}.
Let us give some examples:
\begin{itemize}
	\item Immigration-death processes (see \cite{albanese2005affine}) are defined by a constant birth rate $b$ and a linear death rate $d(x)=dx$. In such case the invariant measure is given by the Poisson distribution $m(x)=e^{-\rho}\frac{\rho^x}{x!}$ for $x \in \N$, where $\rho=\frac{b}{d}$. The orthogonal polynomials are Charlier polynomials of parameter $\rho$, that we denote by $P_n(x)=C_n(x;\rho)$ and they satisfy the self-duality relation
	\begin{equation}\label{sdr}
	P_n(x)=P_x(n)
	\end{equation}
	and the eigenvalues are given by $\lambda_n=-bn$. In particular the polynomials $P_n$ coincide with the family of dual orthogonal polynomials $\widetilde{P}_n$. This kind of process arises as a lattice approximation of the Ornstein-Uhlenbeck process.
	\item Let us consider a linear death rate $d(x)=dx$ and a linear birth rate $b(x)=(x+\beta)b$, with $b,d,\beta>0$ and $b<d$. In such case the birth death process admits state space $E=\N$ and the generator is given by
	\begin{equation*}
	\cG=(\beta b+(b-d)x)\nabla^++dx\Delta.
	\end{equation*}
	Defining $\rho=\frac{b}{d}$, we have that the orthogonal polynomials are Meixner polynomials of parameters $\rho$ and $\beta$, that we denote by $P_n(x)=M_n(x;\rho,\beta)$ and they are orthogonal with respect to the invariant measure
	\begin{equation*}
	m(x)=\frac{(\beta)_x \rho^x}{x! (1-\rho)^\beta},
	\end{equation*}
	where $(\beta)_x=\frac{\Gamma(\beta+x)}{\Gamma(\beta)}$, which is a Pascal (or negative binomial) distribution of parameters $\beta$ and $\rho$. Also Meixner polynomials satisfy the self-duality relation \eqref{sdr} and coincide with their dual polynomials.
	Finally, let us observe that the eigenvalues are given by $\lambda_n=-(d-b)n$. This process is called the Meixner process and arises as lattice approximation of the Cox-Ingersoll-Ross process. Meixner processes are discussed for instance in \cite{karlin1958linear}.
	\item Another case is given by a linear death rate $d(x)=dx$ and a linear decreasing birth rate $b(x)=(N-x)b$ with $b,d>0$ and $N \in \N$. In such case the birth-death process admits finite state space $E=\{0,\dots,N\}$ and the generator is given by
	\begin{equation*}
	\cG=(N b-(b+d)x)\nabla^++dx\Delta.
	\end{equation*}
	Defining $p=\frac{b}{b+d}$ and $q=1-p$ to achieve the invariant distribution given by
	\begin{equation*}
	m(x)=\binom{N}{x}p^xq^{N-x}
	\end{equation*}
	which is a Binomial distribution over $E$. The orthogonal polynomials are Krawtchouk polynomials of parameters $N$ and $p$, that we denote by $P_n(x)=K_n(x;N,p)$ and satisfy the self-duality relation \ref{sdr}. The eigenvalues of the generator are given by $\lambda_n=-n(b+d)$. Let us recall that this is actually a time-continuous version of the Ehrenfest urn model (see, for instance, \cite{karlin1965ehrenfest}).
	\item Another interesting case is given by a quadratic one. Indeed let us consider for some $d>0$ and $\alpha,\beta,N \in \N$
	\begin{align*}
	d(x)=dx(N+\beta+1-x) && b(x)=d[N(\alpha+1)+x(N-1-\alpha)-x^2].
	\end{align*}
	Let us observe that $b(N)=0$, thus the state space of the process is given by $E=\{0,\dots,N\}$. Moreover, its generator is given by
	\begin{equation*}
	\cG=d(N(\alpha+1)-(\beta+\alpha+2)x)\nabla^++dx(N+\beta+1-x)\Delta
	\end{equation*}
	with eigenvalues
	\begin{equation*}
	\lambda_n=-dn[n+1+\alpha+\beta].
	\end{equation*}
	The invariant measure of this birth-death process is an hypergeometric distribution on $E$ given by
	\begin{equation*}
	m(x)=\binom{\alpha+x}{x}\binom{\beta+N-x}{N-x}
	\end{equation*}
	and the orthogonal polynomials are the Hahn polynomials, that we denote by $P_n(x)=H_n(x;\alpha,\beta,N)$. In this case we do not have self-duality relation, but the family of dual Hahn polynomials, that we denote by $\widetilde{P}_n(x)=R_n(x;\alpha,\beta,N)$, is linked to the Hahn polynomials by the relation
	\begin{equation*}
	H_n(x;\alpha,\beta,N)=R_x(n(n+\alpha+\beta+1);\alpha,\beta,N).
	\end{equation*}
	This particular birth-death process is a lattice approximation of the Jacobi process. For such process, we refer directly to \cite{schoutens2012stochastic}.
\end{itemize}
In the examples we have not only the lattice approximations of the \textit{light-tailed} Pearson diffusions, but also another process, which is the continuous-time Ehrenfest urn process. Hence we can observe that with the definition of solvable birth-death process we do not only cover these lattice approximations, but we also gain some other birth-death processes that are not covered in the theory of light-tailed Pearson diffusions (see \cite{leonenko2013fractional,leonenko2017heavy}).\\
Let us observe that for a birth-death process $N(t)$ the forward operator $\cL$ is defined as
\begin{equation*}
\cL f(x)=-\nabla^-((b(\cdot)-d(\cdot))f(\cdot))(x)+\Delta (d(\cdot)f(\cdot))(x).
\end{equation*}
Now let us show what the orthogonal polynomials $P_n(x)$ represent for the forward operator.
\begin{lem}\label{lemM}
	Let $N(t)$ be a solvable birth-death process with forward operator $\cL$, invariant measure $\m$ and associated family of orthonormal polynomials $Q_n(x)$. Then we have for any $n \in E$ and any $x \in E$,
	\begin{equation*}
	\cL_{z \to x}(m(z)Q_n(z))(x)=m(x)\lambda_nQ_n(x).
	\end{equation*}
	\end{lem}
\begin{proof}
	We have, recalling that $\Delta=\nabla^-\nabla^+$ and that $\nabla^-$ is a linear operator,
	\begin{align*}
	\cL_{z \to x} (m(z)Q_n(z))(x)&=-\nabla^-_{z \to x}((b(z)-d(z))m(z)Q_n(z))(x)+\Delta_{z \to x}(d(z)m(z)Q_n(z))(x)\\
	&=\nabla^-_{z \to x}[-(b(z)-d(z))m(z)Q_n(z)+\nabla^+_{y \to z}(d(y)m(y)Q_n(y))(z)](x).
	\end{align*}
	Now, by discrete Leibniz rule on $\nabla^+$, we have
	\begin{align*}
	\cL_{z \to x} (m(z)Q_n(z))(x)&=\nabla^-_{z \to x}[-(b(z)-d(z))m(z)Q_n(z)\\
	&\qquad +Q_n(z)\nabla^+_{y \to z}(d(y)m(y))(z)\\
	&\qquad +d(z+1)m(z+1)\nabla^+Q_n(z)](x)\\
	&=\nabla^-_{z \to x}[Q_n(z)(-(b(z)-d(z))m(z)Q_n(z)+\nabla^+_{y \to z}(d(y)m(y))(z))\\
	&\qquad +d(z+1)m(z+1)\nabla^+Q_n(z)](x)\\
	&=\nabla^-_{z \to x}(d(z+1)m(z+1)\nabla^+Q_n(z))(x).
	\end{align*}
	Now let us use again Leibniz rule on $\nabla^-$ to achieve
	\begin{align*}
	\cL_{z \to x} (m(z)Q_n(z))(x)&=d(x)m(x)\Delta Q_n(x)+\nabla^+Q_n(x)\nabla^-_{z \to x}(d(z+1)m(z+1))(x).
	\end{align*}
	Now let us work with $\nabla^-_{z \to x}(d(z+1)m(z+1))(x)$. We have
	\begin{equation*}
	\nabla^-_{z \to x}(d(z+1)m(z+1))(x)=d(x+1)m(x+1)-d(x)m(x).
	\end{equation*}
	However, by the discrete Pearson equation \eqref{dPe} we obtain
	\begin{equation*}
	d(x+1)m(x+1)-d(x)m(x)=(b(x)-d(x))m(x)
	\end{equation*}
	and then we have
	\begin{equation*}
	\nabla^-_{z \to x}(d(z+1)m(z+1))(x)=(b(x)-d(x))m(x).
	\end{equation*}
	Hence we achieve
	\begin{align*}
	\cL_{z \to x} (m(z)Q_n(z))(x)&=m(x)[(b(x)-d(x))\nabla^+Q_n(x)+d(x)\Delta Q_n(x)]\\
	&=m(x)\cG Q_n(x)=m(x)\lambda_n Q_n(x),
	\end{align*}
	concluding the proof.
	\end{proof}
Thus we have, as a consequence of the discrete Pearson equation, the discrete version of the spectral decomposition for parabolic problems with the generator and the forward operator of a light-tailed Pearson diffusion:
\begin{thm}\label{thm:solv}
	Let $N(t)$ be a solvable birth-death process with state space $E$, generator $\cG$, forward operator $\cL$, invariant measure $\m$ and family of associated orthonormal polynomials $(Q_n)_{n \in E}$. Then the following assertions hold true:
	\begin{itemize}
		\item The transition probability function $p(t,x;y)=\bP(N(t+s)=x|N(s)=y)$ for $x,y \in E$ and $t,s \ge 0$ admit the following spectral representation:
		\begin{equation*}
		p(t,x;y)=m(x)\sum_{n\in E}e^{\lambda_n t}Q_n(x)Q_n(y)
		\end{equation*}
		for any $x,y \in E$ and $t \ge 0$.
		\item If $g \in \ell^2(\m)$ with $g(x)=\sum_{n \in E}g_nQ_n(x)$ then the strong solution of the Cauchy problem
		\begin{equation*}
		\begin{cases}
		\der{u}{t}(t,y)=\cG u(t,y) & t \ge 0, y \in E \\
		u(0,y)=g(y) & y \in E
		\end{cases}
		\end{equation*}
		is given by
		\begin{equation*}
		u(t,y)=\sum_{n \in E}g_ne^{\lambda_n t}Q_n(y)=\sum_{x \in E}p(t,x;y)g(x).
		\end{equation*}
		In particular $p(t,x;y)$ is the fundamental solution of $\der{u}{t}(t,y)=\cG u(t,y)$ and $u$ admits the stochastic interpretation:
		\begin{equation*}
		u(t,y)=\E_y[g(N(t))]
		\end{equation*}
		where $\E_y[\cdot]=\E[\cdot| N(0)=y]$.
		\item If $f/m \in \ell^2(\m)$ with $\frac{f(x)}{m(x)}=\sum_{n \in E}f_nQ_n(x)$ the strong solution of the Cauchy problem
		\begin{equation*}
		\begin{cases}
		\der{v}{t}(t,x)=\cL u(t,x) & t \ge 0, x \in E \\
		v(0,x)=f(x) & x \in E
		\end{cases}
		\end{equation*}
		is given by
		\begin{equation*}
		v(t,x)=m(x)\sum_{n \in E}f_ne^{\lambda_n t}Q_n(x)=\sum_{y\in E}p(t,x;y)f(y).
		\end{equation*}
		In particular $p(t,x;y)$ is the fundamental solution of $\der{v}{t}(t,x)=\cL v(t,x)$ and, if $f \ge 0$ with $\Norm{f}{\ell^1}=1$, $v$ admits the stochastic interpretation:
		\begin{equation*}
		v(t,x)=\bP_f(N(t)=x)
		\end{equation*}
		where $\bP_f$ is the probability measure obtained by conditioning with respect to the fact that $N(0)$ admits distribution $f$.
	\end{itemize}
\end{thm}
From this theorem, it is easy to determine the covariance of any solvable birth-death process in its stationary form. First of all, let us observe that the stationary version of $N(t)$ admits moments of any order. This is obvious if $E$ is finite. To show this when $E=\N_0$, let us first show the following Proposition.
\begin{prop}\label{prop:geomdec}
Let $N(t)$ be a solvable birth-death process with invariant measure $\mm$ and state space $E=\N_0$. Then there exists a constant $\rho<1$ and a state $x_0 \in E$ such that for any $x \ge x_0$ it holds
\begin{equation}\label{eq:geomdec}
m(x)\le \rho^{x-x_0}m(x_0)
\end{equation}
\end{prop}
\begin{proof}
First of all, let us observe that since $b(x)$ and $d(x)$ are polynomials, then $\lim_{x \to +\infty}\frac{b(x)}{d(x+1)}$ always exists. Moreover, we can rewrite the discrete Pearson equation as
\begin{equation*}
m(x+1)=\frac{b(x)}{d(x+1)}m(x).
\end{equation*}
Thus we have that $m$ is well defined if and only if
\begin{equation*}
\sum_{x=1}^{+\infty}\prod_{k=0}^{x}\frac{b(k)}{d(k+1)}<+\infty.
\end{equation*}
It is easy to see that such condition implies $\lim_{x \to +\infty}\frac{b(x)}{d(x+1)}\le 1$. Let us suppose $\lim_{x \to +\infty}\frac{b(x)}{d(x+1)}=1$. This could happen only if $b(x)$ and $d(x)$ are polynomials of the same degree and with the same director coefficient. However, if $b(x)$ and $d(x)$ are polynomials of degree at most $1$, then $\lambda_n=0$ for any $n \ge 1$, that is absurd. Thus we have that $b(x)$ and $d(x)$ are polynomials of degree $2$. However, since $\lambda_n<0$ for any $n \ge 1$, it follows that the coefficient director of $d(x)$ must be negative. However, this means that for $x$ big enough it holds $d(x)<0$, which is absurd. Thus we conclude that
\begin{equation*}
\lim_{x \to +\infty}\frac{b(x)}{d(x+1)}=l<1.
\end{equation*}
Now let us consider $\rho \in (l,1)$. Then there exists a state $x_0 \in E$ such that $\frac{b(x)}{d(x+1)}<\rho$ as $x \ge x_0$. Thus we have
\begin{equation*}
m(x+1)<\rho m(x)
\end{equation*}
for any $x \ge x_0$. Finally, the assertion follows from the previous inequality by induction.
\end{proof}
As a direct consequence of the previous proposition we obtain
\begin{cor}
Let $N(t)$ be a solvable birth-death process with invariant distribution $\mm$ such that $N(0)$ admits $\mm$ as distribution. Then $N(t)$ admits moments of any order.
\end{cor}
Now we can focus on the autocovariance of the process $N(t)$.
\begin{cor}\label{corcov}
	Let $N(t)$ be a solvable birth-death process with invariant measure $\mm$. Then there exists a constant $a_1 \in \R$ such that, for any $t,s \ge 0$
	\begin{equation*}
	\Cov_m(N(t),N(s))=a_1^2e^{\lambda_1|t-s|}.
	\end{equation*}
\end{cor}
\begin{proof}
	First of all let us recall that the stationary version of $N(t)$ admits moments of any order, thus in particular also second order moments and then the autocovariance is well-defined. Since we are supposing that $N(t)$ is stationary, we have, for any $t \ge s$,
	\begin{equation*}
	\Cov_m(N(t),N(s))=\Cov_m(N(t-s)N(0))
	\end{equation*}
	Thus let us consider $t \ge 0$ and let us evaluate $\Cov_m(N(t),N(0))$. To do this, let us rewrite
	\begin{equation*}
	\Cov_m(N(t),N(0))=\E_m[N(t)N(0)]-\E_m[N(t)]\E_m[N(0)]=\E_m[N(t)N(0)]-\E_m[N(0)]^2.
	\end{equation*}
	Now let us first evaluate $\E_x[N(t)]$. Since $\m$ admits second moment then $\iota(x)=x$ is in $\ell^2(\m)$. Moreover, since $\deg(\iota(x))=1$, it can be written as a linear combination of $Q_0=1$ and $Q_1$. Let then
	\begin{equation*}
	\iota(x)=a_0+a_1Q_1(x).
	\end{equation*}
	By the previous theorem we have that
	\begin{equation*}
	\E_x[N(t)]=a_0+a_1e^{\lambda_1 t}Q_1(x)
	\end{equation*}
	(recalling that $\lambda_0=0$ for any solvable birth-death process). Starting from this observation, we have that
	\begin{equation*}
	\E_m[N(t)N(0)]=\sum_{x \in E}xm(x)\E_x[N(t)]=a_0\sum_{x \in E}xm(x)+a_1e^{\lambda_1 t}\sum_{x \in E}xm(x)Q_1(x).
	\end{equation*}
	As we stated before, we can write $x=a_0+a_1Q_1(x)$, thus we have
	\begin{align*}
	\E_m[N(t)N(0)]&=a^2_0\sum_{x \in E}m(x)+a_0a_1\sum_{x \in E}Q_0Q_1(x)m(x)\\&+a_0a_1e^{\lambda_1 t}\sum_{x \in E}Q_0m(x)Q_1(x)+a^2_1e^{\lambda_1 t}\sum_{x \in E}m(x)Q^2_1(x).
	\end{align*}
	By using the orthonormality relation we have
	\begin{align*}
	\E_m[N(t)N(0)]=a^2_0+a^2_1e^{\lambda_1 t}.
	\end{align*}
	Now let us evaluate $\E_m[N(0)]$. We have
	\begin{align*}
	\E_m[N(0)]&=\sum_{x \in E}xm(x)=a_0+a_1\sum_{x \in E}Q_1(x)m(x)\\
	&=a_0+a_1\sum_{x \in E}Q_0Q_1(x)m(x)=a_0.
	\end{align*}
	We finally achieve
	\begin{equation*}
	Cov_m(N(t),N(0))=a_1^2e^{\lambda_1 t}
	\end{equation*}
	concluding the proof.
\end{proof}
\subsection{Classification of solvable birth-death processes}
We can actually improve the result in Proposition \ref{prop:geomdec} by obtaining a complete classification of solvable birth-death processes. Indeed we have the following Proposition.
\begin{prop}
	Let $N(t)$ be a solvable birth-death process with state space $E$. Then one of the following statements holds true:
	\begin{itemize}
		\item $E$ is finite;
		\item $N(t)$ is an immigration-death process;
		\item $N(t)$ is a Meixner process.
	\end{itemize}
In particular, if $(P_n)_{n \in E}$ is the family of orthogonal polynomials associated to $N(t)$, then either $E$ is finite or $(P_n)_{n \in E}$ coincide with its dual family and $P_n(x)=P_x(n)$ for any $x,n \in E$.
\end{prop}
\begin{proof}
	First of all, suppose $E=\N_0$. Then we have $\lim_{x \to +\infty}\frac{b(x)}{d(x+1)}<1$. In particular this implies that $\deg b \le \deg d$. If $\deg d=0$, then also $\deg b=0$ and in such case $\lambda_n=0$ for any $n \in \N$, which is absurd. Thus $\deg d=1,2$. However, we have already seen that if $\deg d=2$, then, since $\lambda_n<0$ for any $n \ge 1$, the director coefficient of $d$ must be negative and this is a contradiction with the fact that $d$ is non negative on $E$. Thus we conclude that if $E=\N_0$, then $\deg d=1$. Moreover, arguing as before, we know that the director coefficient of $d$ must be positive (since if it is negative then $d$ is negative for big values of $x \in E$) and, being also $d(0)=0$, it must hold $d(x)=dx$ for some $d>0$. Now let us consider $b$. Since we want $E=\N_0$, arguing as we did with $d$, we need the director coefficient of $b$ to be positive. Hence we have two possibilities:
	\begin{itemize}
		\item $\deg b=0$, thus $b>0$ is constant and we get an immigration-death process;
		\item $\deg b=1$, thus $b(x)=b(x+\beta)$ for some $b,\beta>0$ and then we get a Meixner process (since also $b<d$ by the condition $\lim_{x \to +\infty}\frac{b(x+\beta)}{d(x+1)}=\frac{b}{d}<1$).
	\end{itemize} 
This concludes the proof.
\end{proof}
As we can see from the previous Proposition, we already discussed as examples the unique two cases in which the state space is countably infinite, that are actually the ones for which the proof of the main results (that will follow) are more articulated. Let us also observe some other particular properties concerning the classification of solvable birth-death processes:
\begin{itemize}
	\item In the class of solvable birth-death processes we have lattice approximations of Pearson diffusions of first spectral category. Pearson diffusions are statistically tractable diffusions (see \cite{forman2008pearson}), however, their spectral behaviour can be distinguished in three different classes (see \cite{leonenko2017heavy}). The first spectral class (the one that contains Pearson diffusions whose generators admit purely discrete spectrum) is composed of the Ornstein-Uhlenbeck, Cox-Ingersoll-Ross and Jacobi diffusions. These diffusions are approximated respectively by immigration-death, Meixner and Hahn birth-death processes;
	\item However, we do not obtain the analogous of Pearson diffusions of the second spectral category: this is due to the fact that to obtain these, we need $d(x)$ to be a polynomial of degree $2$ and with positive director coefficient, which goes in contradiction with the request that the spectrum of the generator of a solvable birth-death process is purely discrete and non-positive together with the fact that the state space is infinite. However, if we reduce the state space to a finite segment of $\N_0$, we can still cover these cases, in which $d(x)$ admits discriminant $\Delta \ge 0$ and $d(0)=0$, by asking that $b(x)$ admits the first root $x_0 \in \N$ and $E=\{0,\dots,x_0\}$. If $\Delta>0$ we also have to ask that the other root of $d(x)$ is negative. In any case, these lattice schemes do not approximate reciprocal Gamma and Fisher-Snedecor diffusions in the support of their invariant measure, but they still provide, in some cases, an approximation scheme for the backward and forward Kolmogorov equations in a subset of the full support;
	\item Even considering a finite segment in $\N_0$, we cannot approximate by a solvable birth-death process the Student distribution, which belongs to the third spectral category: this is due to the fact that to achieve this approximation, $d(x)$ must be a polynomial of degree $2$, with positive director coefficient and negative discriminant, which is in contradiction with the condition $d(0)=0$;
	\item We also get some birth-death processes that are not actual lattice approximation of any Pearson diffusion, such as in the Ehrenfest process case, whose state space is finite but $d(x)$ is a polynomial of degree $1$.
	\item In particular, let us notice that for any solvable birth-death processes $N(t)$ such that the polynomial $b(x-1)-d(x)$ is of degree $1$, the invariant measure $\mm$ is in the Ord family (see \cite{johnson2005univariate}). Indeed if $\deg(b)=\deg(d)=2$, we can suppose that (since $d(0)=0$)
	\begin{equation*}
	d(x)=\widetilde{a}x^2+\widetilde{d}_1x, \qquad b(x)=\widetilde{a}x^2+\widetilde{b}_1x+\widetilde{b}_2.
	\end{equation*}
	By using the relation $m(x-1)=\frac{d(x)}{b(x-1)}m(x)$ and setting
\begin{equation*}
\begin{cases}
k=2\widetilde{a}+\widetilde{d}_1-\widetilde{b}_1 \not =0,\\
a=\frac{\widetilde{a}-\widetilde{b}_1+\widetilde{b}_2}{k},\\
b_0=0,\\
b_1=\frac{\widetilde{d}_1+\widetilde{a}}{k},\\
b_2=\frac{\widetilde{a}}{k},
\end{cases}
\end{equation*}
where $k \not = 0$ since it is the director coefficient of the first degree polynomial $b(x-1)-d(x)$, we get the equation
\begin{equation}\label{eq:Ord}
\frac{\nabla^+m(x-1)}{m(x)}=\frac{a-x}{(a+b_0)+(b_1-1)x+b_2x(x-1)}
\end{equation}
which is the characterizing equation of the Ord family.\\
If $\deg(d)=\deg(b)=1$, then we can suppose that
\begin{equation*}
d(x)=\widetilde{d}_1x, \qquad b(x)=\widetilde{b}_1x+\widetilde{b}_0.
\end{equation*}
To have $\deg(b(x-1)-d(x))=1$, we need $\widetilde{d}_1 \not = \widetilde{b}_1$. Thus, setting
\begin{equation*}
\begin{cases}
k=\widetilde{d}_1-\widetilde{b}_1 \not =0,\\
a=\frac{\widetilde{b}_0-\widetilde{b}_1}{k},\\
b_0=0,\\
b_1=\frac{\widetilde{b}_0}{k},\\
b_2=0,
\end{cases}
\end{equation*}
we still get equation \eqref{eq:Ord}. In particular we get the form
\begin{equation}\label{eq:Katz}
\frac{\nabla^+m(x-1)}{m(x)}=\frac{a-x}{a+(b_1-1)x}
\end{equation}
which is the characterizing equation of the Katz family. Finally, last case is $\deg(d)=1$ and $\deg(b)=0$, that is to say
\begin{equation*}
d(x)=\widetilde{d}_1x\qquad b(x)\equiv\widetilde{b}_0
\end{equation*}
in which the substitution that leads to equation \eqref{eq:Ord} is given by
\begin{equation*}
\begin{cases}
k=\widetilde{d}_1,\\
a=\frac{\widetilde{b}_0}{k},\\
b_0=0,\\
b_1=1,\\
b_2=0.
\end{cases}
\end{equation*}
Even in this case we are actually in the Katz family. Finally, let us observe that in such case the state space has to be infinite (since $b(x)\equiv \widetilde{b}_0>0$) and then we are considering an immigration-death process (and the invariant measure is a Poisson measure). In particular we cover all the distributions of the Katz family (see \cite{johnson2005univariate}).
\item It is also interesting to see that we cover the Poisson, Binomial, Negative Binomial and Hypergeometric invariant distribution cases, which are all in the cumulative Ord family (see \cite{afendras2018orthogonal}).

\end{itemize}
\section{Inverse subordinators and non-local convolution derivatives}\label{Sec3}
Now let us introduce our main object of study. Let us denote by $\BF$ the convex cone of Bernstein functions, that is to say $\Phi \in \BF$ if and only if $\Phi \in C^\infty(\R^+)$, $\Phi(\lambda) \ge 0$ and for any $n \in \N$
\begin{equation*}
(-1)^n\dersup{\Phi}{\lambda}{n}(\lambda)\le 0.
\end{equation*}
In particular it is known that for $\Phi \in \BF$ the following L\'evy-Khintchine representation (\cite{schilling2012bernstein}) is given
\begin{equation}\label{LKrepr}
\Phi(\lambda)=a+b\lambda+\int_0^{+\infty}(1-e^{-\lambda t})\nu(dt)
\end{equation}
where $a,b \ge 0$ and $\nu$ is a L\'evy measure on $\R^+$ such that
\begin{equation}\label{intcond}
\int_0^{+\infty}(1 \wedge t)\nu(dt)<+\infty.
\end{equation}
The triple $(a,b,\nu)$ is called the L\'evy triple of $\Phi$. Also the vice versa can be shown, i.e. for any L\'evy triple $(a,b,\nu)$ such that $\nu$ is a L\'evy measure satisfying the integral condition \eqref{intcond} there exists a unique Bernstein function $\Phi$ such that Equation \eqref{LKrepr} holds. We will say that $\Phi$ is a driftless Bernstein function if $a,b=0$ and $\nu(0,+\infty)=+\infty$. Actually, the definition of driftless Bernstein function only requires $b=0$, but the other two assumptions will be useful in our work.\\
It is also known (see \cite{schilling2012bernstein}) that for each Bernstein function $\Phi \in \BF$ there exists a unique subordinator $\sigma_\Phi=\{\sigma_\Phi(y), y \ge 0\}$ (i. e. an increasing L\'evy process) such that
\begin{equation*}
\E[e^{-\lambda \sigma_\Phi(y)}]=e^{-y\Phi(\lambda)}.
\end{equation*}
In particular we will say that $\sigma_\Phi$ is driftless if $\Phi$ is driftless. For general notion on subordinators we refer to \cite[Chapter $3$]{bertoin1996levy} and \cite{bertoin1999subordinators}. In particular the hypothesis $b=0$ ensure that $\sigma_\Phi$ is a pure jump process, $a=0$ implies that it is not killed and $\nu(0,+\infty)=+\infty$ implies that $\sigma_\Phi$ is strictly increasing (a. s.) and, for each $y>0$, $\sigma_\Phi(y)$ is an absolutely continuous random variable, hence it admits a density $g_\Phi(x;y)$.\\
Let us now fix our driftless Bernstein function $\Phi$ and its associated driftless subordinator $\sigma_\Phi$. Now we can define the inverse subordinator $E_\Phi$ as, for any $t>0$
\begin{equation*}
E_\Phi(t):=\inf\{y \ge 0: \ \sigma_\Phi(y)>t\}.
\end{equation*}
Under our hypotheses, we have that $E_\Phi(t)$ is absolutely continuous for any $t>0$. Let us denote by $f_\Phi(s;t)$ its density. Let us recall (see \cite{meerschaert2008triangular}) that, denoting by $\overline{f}_\Phi(s;\lambda)$ the Laplace transform of $f_\Phi(s;t)$ with respect to $t$,
\begin{equation*}
\overline{f}_\Phi(s;\lambda)=\frac{\Phi(\lambda)}{\lambda}e^{-s\Phi(\lambda)}.
\end{equation*}
Now let us introduce the non-local convolution derivatives (of Caputo type) associated with $\Phi$. Indeed, for $\Phi$ identified by the L\'evy triple $(0,0,\nu)$, let us define the L\'evy tail $\overline{\nu}(t)=\nu(t,+\infty)$. Now let us recall the definition of non-local convolution derivative, defined in \cite{kochubei2011general} and \cite{toaldo2015convolution}.
\begin{defn}
	Let $f:\R^+ \to \R$ be an absolutely continuous function. Then we define the non-local convolution derivative induced by $\Phi$ of $f$ as
	\begin{equation}\label{Cap}
	\dersup{}{t}{\Phi}f(t)=\int_0^t f'(\tau)\overline{\nu}(t-\tau)d\tau.
	\end{equation}
\end{defn}
Let us observe that one can define also the regularized version of the non-local convolution derivative as
\begin{equation}\label{RegCap}
\dersup{}{t}{\Phi}f(t)=\der{}{t}\int_0^t  (f(\tau)-f(0+))\overline{\nu}(t-\tau)d\tau
\end{equation}
observing that it coincides with the previous definition on absolutely continuous functions.\\
It can be shown, by Laplace transform arguments (see, for instance \cite{kochubei2019growth,ascione2020generalized}) or by Green functions arguments (see \cite{kolokol2019mixed}), that the (eigenvalue) Cauchy problem
\begin{equation*}
\begin{cases}
\dersup{}{t}{\Phi} \fe_\Phi(t;\lambda)=\lambda \fe_\Phi(t;\lambda) & t>0\\
\fe_\Phi(0;\lambda)=1
\end{cases}
\end{equation*}
admits a unique solution for any $\lambda>0$ and it is given by $\fe_\Phi(t;\lambda):=\E[e^{\lambda E_\Phi(t)}]$ (hence, in particular, it is a completely monotone function in $\lambda$ for fixed $t$).
Let us recall that if $\Phi(\lambda)=\lambda^\alpha$ for $\alpha \in (0,1)$, then $\overline{\nu}(t)=\frac{t^{-\alpha}}{\Gamma(1-\alpha)}$ and $\dersup{}{t}{\Phi}$ coincides with the fractional Caputo derivative of order $\alpha$. In particular this means that in this case $\fe_{\Phi}(t;\lambda)=E_\alpha(\lambda t^{\alpha})$ where $E_\alpha$ is the one-parameter Mittag-Leffler function defined, for $t \in \R$ as
\begin{equation*}
E_\alpha(t)=\sum_{k=0}^{+\infty}\frac{t^k}{\Gamma(\alpha k+1)}.
\end{equation*}
Let us recall (see \cite{simon2014comparing}) that
\begin{equation*}
E_\alpha(-\lambda t^\alpha)\le \frac{1}{1+\frac{t^\alpha}{\Gamma(1+\alpha)}\lambda}
\end{equation*}
hence it is not difficult to show the following Proposition.
\begin{prop}
	For any $\lambda>0$ it holds
	\begin{equation*}
	\lambda E_\alpha(-\lambda t^\alpha)\le \frac{\Gamma(1+\alpha)}{t^\alpha}.
	\end{equation*}
\end{prop}
The proof is identical to the one of \cite[Lemma $4.2$]{ascione2019fractional}.
We want to achieve a similar bound for any inverse subordinator. This is done by means of the following proposition.
\begin{prop}
	Fix $t>0$. Then there exists a constant $K(t)$ such that
	\begin{equation}\label{unifest}
	\lambda \fe_\Phi(t;-\lambda)\le K(t), \ \forall \lambda \in [0,+\infty).
	\end{equation}
\end{prop}
\begin{proof}
	Let us first recall that $\fe_\Phi(t;-\lambda)=\E[e^{-\lambda E_\Phi(t)}]$, thus it is the Laplace transform of $f_\Phi(s;t)$ with respect to $s$. In particular it is completely monotone in $\lambda$ and $\fe_\Phi(t;0)=1$. Now let us recall that $f_\Phi(0+;t)=\overline{\nu}(t)$ (see, for instance, \cite[Theorem $4.1$]{toaldo2015convolution}). On the other hand, by the initial-value theorem  (see, for instance, \cite[Section $17.8$]{cannon2003dynamics}), we have
	\begin{equation*}
	\lim_{\lambda \to+\infty}\lambda \fe_\Phi(t;-\lambda)=f_\Phi(0+;t)=\overline{\nu}(t)<+\infty.
	\end{equation*}
	Hence we can consider the continuous function $\lambda \in [0,+\infty] \mapsto \lambda\fe_\Phi(t;-\lambda) \in \R^+$ and obtain \eqref{unifest} by Weierstrass theorem.
\end{proof}
Let us give some examples of Bernstein functions and associated subordinators.
\begin{itemize}
	\item We have already referred to the $\alpha$-stable subordinator, i.e. the one we get when we choose $\Phi(\lambda)=\lambda^{\alpha}$ for $\alpha \in (0,1)$. In such case, extensive informations on inverse $\alpha$-stable subordinators are given in \cite{meerschaert2013inverse}. As we stated before, we have in particular $\fe_\Phi(t;\lambda)=E_\alpha(\lambda t^\alpha)$, where $E_\alpha$ is the one-parameter Mittag-Leffler function (see \cite{bingham1971limit}). As a particular property, let us recall that if we denote by $\sigma_\alpha$ the $\alpha$-stable subordinator and $g_\alpha$ the density of the random variable $\sigma_\alpha(1)$, then the inverse $\alpha$-stable subordinator $E_\Phi(t)$ admits density
	\begin{equation*}
	f_\alpha(s;t)=\frac{t}{\beta}s^{-1-\frac{1}{\beta}}g_\alpha(ts^{-\frac{1}{\beta}});
	\end{equation*}
	\item If we fix a constant $\theta>0$ and define $\Phi(\lambda)=(\lambda+\theta)^\alpha-\theta^\alpha$ we obtain the tempered $\alpha$-stable subordinator with tempering parameter $\theta>0$. Denoting by $\sigma_{\alpha,\theta}(t)$ this subordinator, one can show that the density of the subordinator is given by
	\begin{equation*}
	g_{\alpha,\theta}(s;t)=e^{-\theta s+t \theta^\alpha}g_\alpha(s;t)
	\end{equation*}
	where $g_\alpha$ is the density of the $\alpha$-stable subordinator $\sigma_\alpha(t)$. An important property to recall is that the introduction of the tempering parameter implies the existence of all the moments of $\sigma_{\alpha,\theta}(t)$ (while this is not true for $\sigma_\alpha$). Inverse tempered stable subordinators are studied for instance in \cite{kumar2015inverse}. Moreover, it can be shown that the L\'evy tail $\overline{\nu}$ is given by
	\begin{equation*}
	\overline{\nu}(t)=\frac{\alpha \theta^\alpha\Gamma(-\alpha,t)}{\Gamma(1-\alpha)},
	\end{equation*}
	where $\Gamma(\alpha;x)=\int_{x}^{+\infty}t^{\alpha-1}e^{-t}dt$ is the upper incomplete Gamma function;
	\item For $\Phi(\lambda)=\log(1+\lambda^{\alpha})$ as $\alpha \in (0,1)$ we obtain the geometric $\alpha$-stable subordinator. From the form of the Bernstein function associated to the geometric $\alpha$-stable subordinator, one obtains (see \cite[Theorem $2.6$]{vsikic2006potential}) that the density $g_{G,\alpha}$ of the random variable $\sigma_{G,\alpha}(1)$ (where $\sigma_{G,\alpha}(t)$ is the geometric $\alpha$-stable subordinator) satisfies the following asymptotics:
	\begin{align*}
	g_{G,\alpha}(x)\sim \frac{x^{\alpha-1}}{\Gamma(\alpha)} \qquad \mbox{ as }x \to 0^+;\\
	g_{G,\alpha}(x)\sim 2\pi \sin \left(\frac{\alpha \pi}{2}\right) \Gamma(1+\alpha)x^{-\alpha-1} \qquad \mbox{ as }x \to +\infty.
	\end{align*}
	Concerning the L\'evy tail $\overline{\nu}$, it cannot be explicitly expressed, but it has been shown in \cite[Theorem $2.5$]{vsikic2006potential} that it satisfies the following asymptotic relation:
	\begin{equation*}
	\overline{\nu}(t)\sim \frac{t^{-\alpha}}{\Gamma(1-\alpha)} \qquad \mbox{ as }t \to +\infty.
	\end{equation*}
	\item If in the previous example we consider $\alpha=1$ we obtain the Gamma subordinator. In this specific case, one can obtain explicitly the L\'evy tail $\overline{\nu}$ as (see, for instance, \cite{vsikic2006potential} and references therein)
	\begin{equation*}
	\overline{\nu}(t)=\Gamma(0;t).
	\end{equation*}
\end{itemize}
\section{Non-local forward and backward equations}\label{Sec4}
Let us consider $N(t)$ to be a solvable birth-death process with state space $E$ and invariant measure $\mm$ and let us denote by $\cG$ and $\cL$ its backward and forward operators respectively. Let us first focus on the backward equation.
\subsection{Heuristic derivation of the strong solution}
Let us consider a Bernstein function $\Phi$ and the Cauchy problem
\begin{equation}\label{Caueq}
\begin{cases}
\pdsup{}{t}{\Phi}u(t,x)=\cG u(t,x) & t>0, \ x \in E\\
u(0,x)=g(x) & x \in E.
\end{cases}
\end{equation}
Suppose that $g \in \ell^2(\mm)$. Let us consider $(Q_n)_{n \in E}$ the family of orthonormal polynomials associated to $N(t)$. Then we can decompose $g$ as
\begin{equation*}
g(x)=\sum_{n \in E}g_nQ_n(x)
\end{equation*}
for some coefficients $(g_n)_{n \in E}$. Let us suppose we want to find a solution $u(t,x)$ by separation of variables. Thus let us suppose $u(t,x)=T(t)\varphi(x)$. If we substitute this relation in the first equation of \eqref{Caueq} we obtain the following coupled equations
\begin{equation*}
\begin{cases}
\cG \varphi(x)=\lambda \varphi(x) & x \in E \\
\dersup{}{t}{\Phi}T(t)=\lambda T(t) & t>0.
\end{cases}
\end{equation*}
Concerning the first equation, let us observe that we need to set $\varphi(x)=Q_n(x)$ (up to a multiplicative constant) and $\lambda=\lambda_n$ for some $n \in E$. Let us also recall that $\lambda_n<0$. Concerning the second equation we get
\begin{equation*}
T(t)=\fe_{\Phi}(t;\lambda_n).
\end{equation*}
and then we have for some $n \in E$
\begin{equation*}
u(t,x)=Q_n(x)\fe_{\Phi}(t;\lambda_n).
\end{equation*}
Now, we can also have solutions that are linear combinations of $Q_n\fe_\Phi$. Let us suppose that we can consider eventually infinite linear combinations. Then we expect a solution of the form
\begin{equation*}
u(t,x)=\sum_{n \ge 0} u_n Q_n(x)\fe_{\Phi}(t;\lambda_n).
\end{equation*}
for some coefficients $(u_n)_{n \in E}$. Finally, let us observe that
\begin{equation*}
\sum_{n \in E}g_n Q_n(x)=g(x)=u(0,x)=\sum_{n \in E} u_n Q_n(x)
\end{equation*}
then, since the components $(g_n)_{n \in E}$ are uniquely determined, then $u_n=g_n$ for any $n \in E$.\\
Finally, we expect the solution to be of the form
\begin{equation*}
u(t,x)=\sum_{n \in E} g_nQ_n(x)\fe_{\Phi}(t;\lambda_n).
\end{equation*}
Now we want to formalize this reasoning.
\subsection{The backward equation}
Before working in the general case, we need to exploit what will be our fundamental solution. To do this, let us show the following Lemma.
\begin{lem}\label{lem:fundsol}
	Let $N(t)$ be a solvable birth-death process with state space $E$, generator $\cG$, invariant measure $\mm$ and family of associated classical orthogonal polynomials $(P_n)_{n \in E}$. Then the series
	\begin{equation}\label{summation}
	p_\Phi(t,x;y)=m(x)\sum_{n \in E}\fe_\Phi(t;\lambda_n)Q_n(x)Q_n(y),
	\end{equation}
	where $Q_n$ are the normalized orthogonal polynomials, absolutely converges for fixed $t \ge 0$ and $x,y \in E$.
\end{lem}
\begin{proof}
	Let us first observe that if $E$ is finite, then the  summation \eqref{summation} is actually finite.
	Thus let us consider the case in which $E=\N_0$. Let us denote by $\fd_n=\Norm{P_n}{\ell^2}$ and let us recall that the dual polynomials $\widetilde{P}_n$ exhibit orthogonality with respect to the measure $\widetilde{m}(x)=\frac{1}{\fd_x^2}$. However, since, if $E=\N_0$, $N(t)$ is either an immigration-death process or a Meixner process, then the dual polynomials $\widetilde{P}_n$ coincide with the polynomials $P_n$ themselves. We have, by using the self-duality relation $P_n(x)=P_x(n)$,\\
	 \begin{align*}
	 p_\Phi(t,x;y)&=m(x)\sum_{n=0}^{+\infty}\fe_\Phi(t;\lambda_n)Q_n(x)Q_n(y)\\
	 &=m(x)\sum_{n=0}^{+\infty}\widetilde{m}(n)\fe_\Phi(t;\lambda_n)P_n(x)P_n(y)\\
	 &=m(x)\sum_{n=0}^{+\infty}\widetilde{m}(n)\fe_\Phi(t;\lambda_n)P_x(n)P_y(n).
	\end{align*}
	Now let us denote by ${\rm root}(x)$ the set of all the roots of the polynomial $P_x(n)$. These sets are finite with cardinality at most $x$. Thus we can define
	\begin{equation*} 
	 n_0=\lceil\max({\rm root}(x) \cup {\rm root}(y))\rceil+1.
	 \end{equation*}
	 To show the absolute convergence of the series in $p_\Phi(t,x;y)$, we only need to show the absolute convergence of
	 \begin{equation*}
	 \sum_{n=n_0}^{+\infty}\widetilde{m}(n)\fe_\Phi(t;\lambda_n)P_x(n)P_y(n).
	 \end{equation*}
	 Let us observe (see \cite[Table $2.3$]{nikiforov1991classical}) that the sign of the director coefficient of $P_x$ (both in the Charlier than in the Meixner case) depends on the parity of $x$. In particular we have that for any $n \ge n_0$ it holds $\sign(P_x(n)P_y(n))=(-1)^{x+y}$. So we get
	 \begin{align*}
	 \sum_{n=n_0}^{+\infty}&\left|\widetilde{m}(n)\fe_\Phi(t;\lambda_n)P_x(n)P_y(n)\right|\\&\le (-1)^{x+y}\sum_{n=n_0}^{+\infty}\widetilde{m}(n)P_x(n)P_y(n),
	 \end{align*}
	 where we have to observe that $\fe_\Phi(t,\lambda_n)\le 1$ since $\lambda_n \le 0$. As before, the series at the right-hand side converges if and only if
	 \begin{equation*}
	 \sum_{n=0}^{+\infty}\widetilde{m}(n)P_x(n)P_y(n)
	 \end{equation*}
	 converges. However, by the dual orthogonal relation and the self-duality relation of Charlier and Meixner polynomials, we achieve
	 \begin{equation*}
	 \sum_{n=0}^{+\infty}\widetilde{m}(n)P_x(n)P_y(n)=\frac{1}{m(x)}\delta_{x,y}
	 \end{equation*}
	 for any $x,y \in \N$, concluding the proof.
\end{proof}
Before exploiting the strong solution, let us show the total convergence of some auxiliary series of functions.
\begin{lem}\label{lem:convseries}
		Let $N(t)$ be a solvable birth-death process with state space $E=\N_0$, generator $\cG$, invariant measure $\mm$ and family of associated classical orthogonal polynomials $(P_n)_{n \ge 0}$. Let $g \in \ell^2(\mm)$ such that $g(x)=\sum_{n \ge 0}g_nQ_n(x)$ for $x \in E$ and some constants $(g_n)_{n \ge 0}$, where $(Q_n)_{n \ge 0}$ are the normalized orthogonal polynomials. Then
		\begin{enumerate}
			\item For any $x \in E$  it holds  $\sum_{n \ge 0}|g_nQ_n(x)| \le \frac{ \Norm{g}{\ell^2(\mm)}}{\sqrt{m(x)}}$;
			\item For any fixed $x \in E$ the sum $\sum_{n \ge 0}\fe_\Phi(t,\lambda_n)g_nQ_n(x)$ totally converges for $t \in [0,+\infty)$;
			\item For any fixed $x \in E$ and $T_1>0$ the series $\sum_{n \ge 0}\lambda_n\fe_\Phi(t,\lambda_n)g_nQ_n(x)$ totally converges for $t \in [T_1,+\infty)$.
		\end{enumerate}
\end{lem}
\begin{proof}
	Let us show property $(1)$. Since $g(x)=\sum_{n \ge 0}g_nQ_n(x)$ and $Q_n$ is an orthonormal basis of $\ell^2(\mm)$, it holds $\sum_{n \ge 0}g_n^2=\Norm{g}{\ell^2(\mm)}^2$. In particular it holds, by Cauchy-Schwartz inequality and self-duality relation for Charlier and Meixner polynomials
	\begin{align*}
	\sum_{n\ge 0}|g_nQ_n(x)|&=\sum_{n \ge 0}\sqrt{\widetilde{m}(n)}|g_nP_n(x)|\\
	&\le \left(\sum_{n\ge 0}\widetilde{m}(n)P^2_n(x)\right)^{\frac{1}{2}}\Norm{g}{\ell^2(\mm)}\\
	&= \left(\sum_{n\ge 0}\widetilde{m}(n)P^2_x(n)\right)^{\frac{1}{2}}\Norm{g}{\ell^2(\mm)}=\frac{\Norm{g}{\ell^2(\mm)}}{\sqrt{m(x)}}.
	\end{align*}
	To show property $(2)$, let us just observe that $\fe_\Phi(t,\lambda_n)\le 1$ since $\lambda_n\le 0$ and then
	\begin{equation*}
	\sum_{n \ge 0}|\fe_\Phi(t,\lambda_n)g_nQ_n(x)|\le\sum_{n \ge 0}|g_nQ_n(x)| 
	\end{equation*}
	where the series on the right-hand side is convergent and independent of $t\ge 0$.\\
	Concerning property $(3)$, by Equation \eqref{unifest} we obtain for some constant $K(T_1)>0$, since $\fe_\Phi(t,\lambda_n)$ is decreasing,
	\begin{equation*}
	\sum_{n \ge 0}|\lambda_n\fe_\Phi(t,\lambda_n)g_nQ_n(x)|\le \sum_{n \ge 0}|\lambda_n\fe_\Phi(T_1,\lambda_n)g_nQ_n(x)| \le  K(T_1)\sum_{n \ge 0}|g_nQ_n(x)|,
	\end{equation*}
	concluding the proof.
	\end{proof}
Now we are ready to show that for the initial datum $g \in \ell^2(\mm)$ our backward problem admits a solution.
\begin{thm}\label{thm:ssCb}
	Let $N(t)$ be a solvable birth-death process with state space $E$, generator $\cG$, invariant measure $\mm$ and family of associated classical orthogonal polynomials $(P_n)_{n \in E}$. Let $g \in \ell^2(\mm)$ such that $g(x)=\sum_{n \in E}g_nQ_n(x)$ for $x \in E$ and some constants $(g_n)_{n \in E}$, where $(Q_n)_{n \in E}$ are the normalized orthogonal polynomials. Then the Cauchy problem
	\begin{equation}\label{PhiCaupback}
	\begin{cases}
	\pdsup{u}{t}{\Phi}(t,y)=\cG u(t,y) & t>0, \ y \in E \\
	u(0,y)=g(y) & y \in E
	\end{cases}
	\end{equation}
	admits a unique strong solution of the form
	\begin{equation}\label{ssCb}
	u(t,y)=\sum_{n \in E}\fe_\Phi(t,\lambda_n)g_nQ_n(y).
	\end{equation}
	In particular $\sup_{t \ge 0}\Norm{u(t,\cdot)}{\ell^2(\mm)}\le \Norm{g}{\ell^2(\mm)}$.\\
	Finally $p_\Phi(t,x;y)$ is the fundamental solution of \eqref{PhiCaupback}, in the sense that it is the strong solution of \eqref{PhiCaupback} for $g(y)=\delta_x(y)$ and for any $g \in \ell^2(\mm)$ it holds
	\begin{equation*}
	u(t,y)=\sum_{x \in E}p_\Phi(t,x;y)g(x).
	\end{equation*}
\end{thm}
Before giving the proof of the Theorem, let us state formally what we mean as strong solution.
\begin{defn}
	A function $u(t,y)$ is a strong solution of the Cauchy problem \eqref{PhiCaupback} if:
	\begin{itemize}
		\item $\pdsup{u}{t}{\Phi}(t,y)$ exists for any $t>0$ and $y \in E$;
		\item The equations in \eqref{PhiCaupback} hold pointwise;
		\item $u(t;\cdot) \in C([0,+\infty);\ell^2(\mm))$;
		\item $\pdsup{u}{t}{\Phi}(t;\cdot) \in C((0,+\infty);\ell^2(\mm))$.
	\end{itemize}
\end{defn}
\begin{proof}[Proof of Theorem \ref{thm:ssCb}]
	Let us first show that $u(t,y)$ in the form of \eqref{ssCb} is a strong solution for \eqref{PhiCaupback}.\\
	First of all, let us recall that, by definition of $\fe_\Phi$ and $Q_n$, it holds
	\begin{multline*}
	\cG [\fe_\Phi(t,\lambda_n)Q_ng_n](x)= \fe_\Phi(t,\lambda_n)g_n\cG Q_n(x)\\=\lambda_n \fe_\Phi(t,\lambda_n)g_nQ_n(x)=\pdsup{}{t}{\Phi}\fe_\Phi(t,\lambda_n)g_nQ_n(x).
	\end{multline*}
	Now let us observe that if $E$ is finite then $n_E \in \N$ and we have that $u(t,y)$ is a strong solution of \eqref{PhiCaupback} just by linearity of the involved operators.\\
	Let us now suppose that $E$ is countably infinite. First of all, let us show that the series \eqref{ssCb} converges in $\ell^2(\mm)$. To do this, define for $N \in \N$
	\begin{equation*}
	u_N(t,y)=\sum_{n=0}^{N}\fe_\Phi(t,\lambda_n)g_nQ_n(y).
	\end{equation*}	
	Now consider $N<M$ in $\N$ and observe that, being $\lambda_n\le 0$ and $\fe_\Phi(t,\lambda_n)\le 1$, it holds
	\begin{equation*}
	\Norm{u_N(t,\cdot)-u_M(t,\cdot)}{\ell^2(\mm)}^2\le \sum_{n=N}^{M}g_n^2
	\end{equation*}
	thus, by Cauchy's criterion, we know that the series converges in $\ell^2(\mm)$.\\
	Now let us denote
	\begin{equation*}
	I_\nu(t)=\int_0^t\overline{\nu}(\tau)d\tau
	\end{equation*}
	that is increasing and non-negative.\\
	Let us then observe that
	\begin{equation*}
	\int_0^t(u(\tau,y)-u(0+,y))\overline{\nu}(t-\tau)d\tau=\int_0^t(u(\tau,y)-u(0+,y))dI_\nu(t-\tau).
	\end{equation*}
	Since we have shown in Lemma \ref{lem:convseries} that the series defining $u(t,y)$ totally converges for fixed $y \in E$, then we can use \cite[Theorem $7.16$]{rudin1964principles} to write:
	\begin{equation}\label{pass1Cb}
	\int_0^t(u(\tau,y)-u(0+,y))\overline{\nu}(t-\tau)d\tau=\sum_{n=0}^{+\infty}\left(\int_0^t(\fe_\Phi(\tau,\lambda_n)-1)\overline{\nu}(t-\tau)d\tau\right)Q_n(y)g_n.
	\end{equation}
	As next step, we want to differentiate under the series sign. However, we have to show uniform convergence for $t$ in any compact set $[T_1,T_2]$ of the series of the derivatives to use \cite[Theorem $7.17$]{rudin1964principles}. However, recalling \eqref{RegCap}, one has
	\begin{align*}
	\sum_{n=0}^{+\infty}\pd{}{t}\left(\int_0^t(\fe_\Phi(\tau,\lambda_n)-1)\overline{\nu}(t-\tau)d\tau\right)Q_n(y)g_n&=\sum_{n=0}^{+\infty}\pdsup{}{t}{\Phi}\fe_\Phi(t,\lambda_n)Q_n(y)g_n\\&=\sum_{n=0}^{+\infty}\lambda_n\fe_\Phi(t,\lambda_n)Q_n(y)g_n
	\end{align*}
	that totally converges by statement $(3)$ of Lemma \ref{lem:convseries}.\\
	Hence we obtain, differentiating on both sides in \eqref{pass1Cb},
	\begin{equation*}
	\pdsup{}{t}{\Phi}u(t,y)=\sum_{n=0}^{+\infty}\pdsup{}{t}{\Phi}\fe_\Phi(t,\lambda_n)Q_n(y)g_n=\sum_{n=0}^{+\infty}\fe_\Phi(t,\lambda_n)\cG Q_n(y)g_n.
	\end{equation*}
	Now let us recall that $\cG=(b(x)-d(x))\nabla^++d(x)\Delta$, hence we have to show that we can exchange $\nabla^+$ with the sign of series. This is due to the commutative property of totally convergent series. Indeed we have
	\begin{align*}
	\nabla^+\sum_{n=0}^{+\infty}\fe_\Phi(t,\lambda_n) Q_n(y)g_n&=\sum_{n=0}^{+\infty}\fe_\Phi(t,\lambda_n) Q_n(y+1)g_n-\sum_{n=0}^{+\infty}\fe_\Phi(t,\lambda_n) Q_n(y)g_n\\
	&=\lim_{N \to +\infty}\left(\sum_{n=0}^{N}\fe_\Phi(t,\lambda_n) Q_n(y+1)g_n-\sum_{n=0}^{N}\fe_\Phi(t,\lambda_n) Q_n(y)g_n\right)\\
	&=\lim_{N \to +\infty}\sum_{n=0}^{N}\fe_\Phi(t,\lambda_n) \nabla^+Q_n(y)g_n=\sum_{n=0}^{+\infty}\fe_\Phi(t,\lambda_n) \nabla^+Q_n(y)g_n
	\end{align*}
	where all the passages are justified by the fact that the two series \linebreak $\sum_{n=0}^{+\infty}\fe_\Phi(t,\lambda_n) Q_n(y+1)g_n$, $\sum_{n=0}^{+\infty}\fe_\Phi(t,\lambda_n) Q_n(y)g_n$ both totally converge by Lemma \ref{lem:convseries}. The same holds for $\Delta$. Thus we finally have
	\begin{equation*}
	\pdsup{}{t}{\Phi}u(t,y)=\sum_{n=0}^{+\infty}\fe_\Phi(t,\lambda_n)\cG Q_n(y)g_n=\cG u(t,y)
	\end{equation*}
	for any $t>0$. Concerning the initial datum, we have
	\begin{equation*}
	u(0,y)=\sum_{n\in E}g_nQ_n(y)=g(y).
	\end{equation*}
	Now let us show strong continuity of $u$ in $[0,+\infty)$ and of $\frac{\partial^\Phi}{\partial t^\Phi}u$ in $(0,+\infty)$. These properties are obvious as $E$ is finite, thus let us suppose $E=\N_0$. Concerning the continuity of $u$, let us show it in $0^+$, since for any other $t \in (0,+\infty)$ the proof is analogous. We have
	\begin{equation*}
	\Norm{u(t,\cdot)-g(\cdot)}{\ell^2(\mm)}^2=\sum_{n=1}^{+\infty}(1-\fe_\Phi(t,\lambda_n))^2g_n^2.
	\end{equation*}
	Now fix $\varepsilon>0$. Since $(g_n)_{n \ge 0}$ belongs to $\ell^2$, there exists $n(\varepsilon)\ge 0$ such that $\sum_{n=n(\varepsilon)}^{+\infty}g_n^2\le \varepsilon$. By using the fact that $(1-\fe_\Phi(t,\lambda))^2\le 1$ for any $\lambda<0$, we get
	\begin{equation*}
	\Norm{u(t,\cdot)-g(\cdot)}{\ell^2(\mm)}^2=\sum_{n=1}^{n(\varepsilon)}(1-\fe_\Phi(t,\lambda_n))^{2}g_n^2+\varepsilon.
	\end{equation*}
	Sending $t \to 0^+$ (using the fact that $\fe_\Phi(t,\lambda_n)$ is continuous in $t$ and that the first summation is finite) and then $\varepsilon \to 0^+$ we obtain strong continuity of $u$. Let us discuss the continuity of $\frac{\partial^\Phi}{\partial t^\Phi}u$ in $(0,+\infty)$. To do this, let us consider $t_0 \in [t_1,t_2]$ with $t_1>0$. We have, for any $t \in [t_1,t_2]$, arguing as before and using \eqref{unifest},
	\begin{align*}
	\Norm{\pdsup{}{t}{\Phi}u(t,\cdot)-\pdsup{}{t}{\Phi}u(t_0,\cdot)}{\ell^2(\mm)}^2&=\Norm{\sum_{n \in E}\lambda_n(\fe_\Phi(t,\lambda_n)-\fe_\Phi(t_0,\lambda_n))Q_n(\cdot)g_n}{\ell^2(\mm)}^2\\
	&=\sum_{n=0}^{+\infty}\lambda_n^2(\fe_\Phi(t,\lambda_n)-\fe_\Phi(t_0,\lambda_n))^2g^2_n\\
	&\le\sum_{n=0}^{n(\varepsilon)}\lambda_n^2(\fe_\Phi(t,\lambda_n)-\fe_\Phi(t_0,\lambda_n))^2g^2_n+K(t_1)\varepsilon.
	\end{align*}
	Thus, sending $t \to t_0$ (observing that the first sum is finite) and then $\varepsilon \to 0^+$ we obtain the desired continuity. Uniqueness follows easily from the fact that $(Q_n)_{n \ge 0}$ is an orthonormal system in $\ell^2(\mm)$ hence the coefficients are unique.\\
	Now let us show the bound of the $\ell^2(\mm)$ norm. We have
	\begin{equation*}
	\Norm{u(t,\cdot)}{\ell^2(\mm)}^2=\sum_{n=0}^{+\infty}\fe_\Phi^2(t,\lambda_n)g_n^2\le\sum_{n=0}^{+\infty}g_n^2=\Norm{g}{\ell^2(\mm)}.
	\end{equation*}
	Now let us consider a function $g \in \ell^2(\mm)$. Then we have
	\begin{align*}
	\sum_{x \in E}p_\Phi(t,x;y)g(x)&=\sum_{x \in E}m(x)\left(\sum_{n \in E}\fe_\Phi(t,\lambda_n)Q_n(x)Q_n(y)\right)g(x)\\&=\sum_{n\in E}Q_n(y)\fe_\Phi(t,\lambda_n)\sum_{x \in E}m(x)Q_n(x)g(x)\\
	&=\sum_{n \in E}Q_n(y)\fe_\Phi(t,\lambda_n)g_n=u(t,y)
	\end{align*}
	where we could exchange the order of the series since or $E$ is finite, and then the sums are finite, or $E$ is countably infinite and all the series involved are totally convergent in compact sets containing $t$.\\
	Finally, let us observe that if for some $z \in E$ $g(x)=\delta_z(x)$, then
	\begin{align*}
	u(t,y)=\sum_{x \in E}p_\Phi(t,x;y)\delta_z(x)=p_\Phi(t,z;y),
	\end{align*}
	concluding the proof.
\end{proof}
\subsection{The forward equation}
Now let us apply the same strategy to study the Cauchy problem associated with $\cL$.
\begin{thm}\label{thm:ssCf}
	Let $N(t)$ be a solvable birth-death process with state space $E$, forward operator $\cL$, invariant measure $\mm$ and family of associated classical orthogonal polynomials $(P_n)_{n \in E}$. Let $f/m \in \ell^2(\mm)$ such that $f(x)=m(x)\sum_{n \in E}f_nQ_n(x)$ for $x \in E$ and some constants $(f_n)_{n \in E}$, where $(Q_n)_{n \in E}$ are the normalized orthogonal polynomials. Then the Cauchy problem
	\begin{equation}\label{PhiCaupfor}
	\begin{cases}
	\pdsup{v}{t}{\Phi}(t,x)=\cL v(t,x) & t>0, \ x \in E \\
	v(0,x)=f(x) & x \in E
	\end{cases}
	\end{equation}
	admits a unique strong solution of the form
	\begin{equation}\label{ssCf}
	v(t,x)=m(x)\sum_{n \in E}\fe_\Phi(t,\lambda_n)f_nQ_n(x),
	\end{equation}
	such that $\sup_{t \ge 0}\Norm{v(t,\cdot)}{\ell^2(\mm)}\le \Norm{f/m}{\ell^2(\mm)}$.
	Finally $p_\Phi(t,x;y)$ is the fundamental solution of \eqref{PhiCaupback}, in the sense that it is the strong solution of \eqref{PhiCaupback} for $f(x)=\delta_y(x)$ and for any $f/m \in \ell^2(\mm)$ it holds
	\begin{equation*}
	v(t,x)=\sum_{y \in E}p_\Phi(t,x;y)f(y).
	\end{equation*}
\end{thm}
\begin{proof}
	Let us first observe, by Lemma \ref{lemM}
	\begin{multline*}
	\cL m(x)\fe_\Phi(t,\lambda_n)f_nQ_n(x)= \fe_\Phi(t,\lambda_n)f_n\cL_{z \to x}(m(z)Q_n(z))(x)\\=\lambda_n\fe_\Phi(t,\lambda_n)f_nm(x)Q_n(x)=\pdsup{}{t}{\Phi}m(x)\fe_\Phi(t,\lambda_n)f_nQ_n(x).
	\end{multline*}
	Thus, the exact same strategy of the proof of Theorem \ref{thm:ssCb} leads to formula \eqref{ssCf} and uniqueness follows as before. Concerning the estimate on the norm, it holds, since $\mm$ is a probability measure and then ${m(x)\le 1}$,
	\begin{align*}
	\Norm{v(t,x)}{\ell^2(\mm)}&=\sum_{x \in E}m(x)\left(m(x)\sum_{n \in E}\fe_\Phi(t,\lambda_n)f_nQ_n(x)\right)^2\\
	&\le \sum_{x \in E}m(x)\left(\sum_{n\in E}f_n\fe_\Phi(t,\lambda_n)Q_n(x)\right)^2\\
	&=\sum_{n\in E}f^2_n\fe^2_\Phi(t,\lambda_n)\le \Norm{f/m}{\ell^2(\mm)}.
	\end{align*}
	Moreover, let us observe that
	\begin{align*}
	\sum_{y \in E}p_\Phi(t,x;y)f(y)&=m(x)\sum_{y \in E}\left(\sum_{n \in E}\fe_\Phi(t,\lambda_n)Q_n(x)Q_n(y)\right)\frac{f(y)}{m(y)}m(y)\\
	&=m(x)\sum_{n \in E}\fe_\Phi(t,\lambda_n)Q_n(x)\sum_{y \in E}Q_n(y)\frac{f(y)}{m(y)}m(y)\\
	&=m(x)\sum_{n \in E}\fe_\Phi(t,\lambda_n)f_nQ_n(x)=v(t,x).
	\end{align*}
	Finally, observe that for some fixed $z \in E$, $f(y)=\delta_z(y)$. Then obviously $f/m \in \ell^2(\mm)$ and then we have
	\begin{equation*}
	v(t,x)=\sum_{y \in E}p_\Phi(t,x;y)f(y)=\sum_{y \in E}p_\Phi(t,x;y)\delta_z(y)=p_\Phi(t,x;z)
	\end{equation*}
	concluding the proof.
\end{proof}
\begin{rmk}
	One obtains also the following estimate on the norm:
	\begin{equation*}
	\sup_{t \ge 0}\Norm{v(t,\cdot)/m(\cdot)}{\ell^2(\mm)}\le \Norm{f/m}{\ell^2(\mm)}.
	\end{equation*}
\end{rmk}
Next step is to identify some processes such that $p_\Phi(t,x;y)$ is its \textit{transition probability function} (in some sense) and then give some stochastic representation of the strong solutions of the Cauchy problems \eqref{ssCb} and \eqref{ssCf}.
\section{Non-local solvable birth-death processes}\label{Sec5}
Let us now consider a solvable birth-death process $N$ and the subordinator $\sigma_\Phi$ associated to the Bernstein function $\Phi$, with its inverse $E_\Phi$. Let us suppose $N$ and $E_\Phi$ are independent.
\begin{defn}
	The non-local solvable birth-death process induced by $N$ and $\Phi$ is defined as
	\begin{equation}
	N_\Phi(t):=N(E_\Phi(t))
	\end{equation}
	for any $t \ge 0$, where $E_\Phi(t)$ is independent of $N(t)$. Moreover, we define its transition probability function
	\begin{equation*}
	p_\Phi(t,x;y):=\bP(N_\Phi(t)=x|N_\Phi(0)=y).
	\end{equation*}
\end{defn}
We used the same notation for the transition probability function and the fundamental solution of the Cauchy problems \eqref{ssCb} and \eqref{ssCf}. Indeed, we can show the following result.
\begin{thm}
	The transition probability function $p_\Phi(t,x;y)$ coincides with the series defined in Lemma \ref{lem:fundsol}.
	\end{thm}
\begin{proof}
	Let us first observe that $N_\Phi(0)=N(0)$ almost surely (since $E_\Phi(0)=0$ almost surely), thus, by a simple conditioning argument, we have
	\begin{equation*}
	p_\Phi(t,x;y)=\int_0^{+\infty}p(s,x;y)f_\Phi(s,t)ds.
	\end{equation*}
	Now let us recall, by Theorem \ref{thm:solv}, that
	\begin{equation*}
	p(s,x;y)=m(x)\sum_{n \in E}e^{\lambda_n t}Q_n(x)Q_n(y)
	\end{equation*}
	thus we have
	\begin{equation*}
	p_\Phi(t,x;y)=m(x)\int_0^{+\infty}\sum_{n \in E}e^{\lambda_n s}Q_n(x)Q_n(y)f_\Phi(s,t)ds.
	\end{equation*}
	If $E$ is finite, we conclude the proof. So let us suppose $E=\N_0$.
	Recalling the notation in the proof of Lemma \ref{lem:fundsol}, let us define
	\begin{equation*} 
	n_0=\lceil\max({\rm root}(x) \cup {\rm root}(y))\rceil+1.
	\end{equation*} 
	and let us split the summation in two parts. We have
	\begin{align*}
	p_\Phi(t,x;y)&=m(x)\int_0^{+\infty}\sum_{n=0}^{n_0}Q_n(x)Q_n(y)e^{\lambda_n s}f_\Phi(s,t)ds\\
	&+m(x)\int_0^{+\infty}\sum_{n={n_0+1}}^{+\infty}Q_n(x)Q_n(y)e^{\lambda_n s}f_\Phi(s,t)ds.
	\end{align*}
	Now let us observe that we can exchange the integral sign with the first summation by linearity of the integral, while in the second summation this can be done by Fubini's theorem, being the integrands of fixed sign (exactly the sign is determined by $(-1)^{x+y}$ since $Q_n(x)Q_n(y)=\widetilde{m}(n)P_x(n)P_y(n)$). Thus we have
	\begin{align*}
	p_\Phi(t,x;y)&=m(x)\sum_{n=0}^{n_0}Q_n(x)Q_n(y)\int_0^{+\infty}e^{\lambda_n s}f_\Phi(s,t)ds\\
	&+m(x)\sum_{n={n_0+1}}^{+\infty}Q_n(x)Q_n(y)\int_0^{+\infty}e^{\lambda_n s}f_\Phi(s,t)ds\\
	&=m(x)\sum_{n=0}^{+\infty}Q_n(x)Q_n(y)\int_0^{+\infty}e^{\lambda_n s}f_\Phi(s,t)ds.
	\end{align*}
	Finally, let us recall that, by definition
	\begin{equation*}
	\int_0^{+\infty}e^{\lambda_n s}f_\Phi(s,t)ds=\E[e^{\lambda_n E_\Phi(t)}]=\fe_\Phi(t,\lambda_n),
	\end{equation*}
	concluding the proof.\\
\end{proof}
By using this result, we can achieve some stochastic representation of the solutions of the Cauchy problems \eqref{ssCb} and \eqref{ssCf}. Indeed we have the following result.
\begin{prop}\label{prop:stocrap}
	Let $N_\Phi(t)$ be the non-local solvable birth-death process associated to $N(t)$ and $\Phi$. Suppose $N(t)$ admits state space $E$. Then
	\begin{enumerate}
		\item For $g \in \ell^2(\mm)$ the function $u(t,y)=\E_y[g(N_\Phi(t))]$ (where $\E_y[\cdot]=\E[\cdot|N_\Phi(0)=y]$) is strong solution of \eqref{ssCb};
		\item For $f/m \in \ell^2(\mm)$ such that $f \ge 0$ and  $\sum_{x \in E}f(x)=1$, denoting by $\bP_{f}$ the probability measure obtained by $\bP$ conditioning with the fact that $N_\Phi(0)$ admits distribution $f$, then the function $v(t,x)=\bP_{f}(N_\Phi(t)=x)$ is strong solution of \eqref{ssCf}.
	\end{enumerate}
\end{prop}
\begin{proof}
	To show assertion $1$, let us observe that
	\begin{equation*}
	u(t,y)=\E_y[g(N_\Phi(t))]=\sum_{x \in E}g(x)p_\Phi(t,x;y),
	\end{equation*}
	obtaining that $u(t,y)$ is the strong solution of \eqref{PhiCaupback} by Theorem \ref{thm:ssCb}.\\
	Concerning assertion $2$, we have
	\begin{equation*}
	v(t,x)=\bP_f(N_\Phi(t)=x)=\sum_{y \in E}f(y)p_\Phi(t,x;y)
	\end{equation*}
	obtaining that $v(t,x)$ is the strong solution of \eqref{PhiCaupfor} by Theorem \ref{thm:ssCf}.
\end{proof}
Finally, this proposition allows us to obtain the invariant distribution of $N_\Phi(t)$ and show that it is also the limit distribution.
\begin{cor}
	Let $N_\Phi(t)$ be the non-local solvable birth-death process associated to $N(t)$ and $\Phi$. Suppose $N(t)$ admits state space $E$. Then
	\begin{enumerate}
		\item If $N_\Phi(0)$ admits distribution $\mm$, then $N_\Phi(t)$ admits distribution $\mm$ for any $t \ge 0$;
		\item If $N_\Phi(0)$ admits distribution $f$ such that $f/m \in \ell^2(\mm)$, then \linebreak $\lim_{t \to +\infty}\bP_f(N_\Phi(t)=x)=m(x)$.
	\end{enumerate}
\end{cor}
\begin{proof}
	Let us first show property $(1)$. To do this, let us observe that $1 \in \ell^2(\mm)$ since $\mm$ is a probability measure. Thus, recall, by Proposition \ref{prop:stocrap}, that $v(t,x)=\bP_{\mm}(N_\Phi(t)=x)$ is a strong solution of \eqref{PhiCaupfor}. Now we need to determine $m_n$ such that $\sum_{n \in E}m_nQ_n(x)=1$. Let us recall that $Q_0(x)=1$ while $\deg(Q_n(x))=n$ for any $n=1,\dots,n_E$, thus we have $m_0=1$ and $m_n=0$ for any $n=1,\dots,n_E$. Moreover, let us recall that $\lambda_0=0$, since $1 \in Ker(\cG)$. Hence we have, by Theorem \ref{thm:ssCf}
	\begin{equation*}
	v(t,x)=m(x)\sum_{n \in E}\fe_\Phi(t,\lambda_n)Q_n(x)m_n=m(x)
	\end{equation*}
	and $N_\Phi(t)$ admits $\mm$ as distribution.\\
	Now let us suppose $N_\Phi(0)$ admits $f$ as distribution with $f/m \in \ell^2(\mm)$. By Proposition \ref{prop:stocrap} we have that $v(t,x)=\bP_f(N_\Phi(t)=x)$ is a strong solution of \eqref{PhiCaupfor} hence it holds
	\begin{equation*}
	v(t,x)=m(x)\sum_{n \in E}\fe_\Phi(t,\lambda_n)Q_n(x)f_n=m(x)f_0+m(x)\sum_{\substack{n \in E \\ n\ge 1}}\fe_\Phi(t,\lambda_n)Q_n(x)f_n.
	\end{equation*}
	Let us determine $f_0$. We have, by definition of scalar product in $\ell^2(\mm)$,
	\begin{equation*}
	f_0=\sum_{x \in E}m(x)\frac{f(x)}{m(x)}Q_0=\sum_{x \in E}f(x)=1
	\end{equation*}
	hence we have
	\begin{equation*}
	v(t,x)=m(x)+m(x)\sum_{\substack{n \in E \\ n\ge 1}}\fe_\Phi(t,\lambda_n)Q_n(x)f_n.
	\end{equation*}
	Now let us consider the summation part. First of all, let us recall that $\mm$ is a probability measure, hence $m(x)\le 1$. We have
	\begin{equation*}
	\sum_{x \in E}m(x)f^2(x)\le \sum_{x \in E}m(x)\frac{f^2(x)}{m^2(x)}=\Norm{f/m}{\ell^2(\mm)}
	\end{equation*}
	hence $f \in \ell^2(\mm)$. By Lemma \ref{lem:convseries} we know that $\sum_{\substack{n \in E \\ n\ge 1}}\fe_\Phi(t,\lambda_n)Q_n(x)f_n$ totally converges, hence we can take the limit as $t \to +\infty$ under the summation sign.\\
	Now let us observe that $\lim_{t \to +\infty}E_\Phi(t)=+\infty$ almost surely. On the other hand, we have $e^{\lambda_n E_\Phi(t)}\le 1$, thus we can use monotone convergence theorem to achieve
	\begin{equation*}
	\lim_{t \to +\infty}\fe_\Phi(t,\lambda_n)=\E[\lim_{t \to +\infty}e^{\lambda_n E_\Phi(t)}]=0.
	\end{equation*}
	Finally, by dominated convergence theorem, we have
	\begin{align*}
	\lim_{t \to +\infty}v(t,x)&=m(x)+m(x)\lim_{t \to +\infty}\sum_{\substack{n \in E \\ n\ge 1}}\fe_\Phi(t,\lambda_n)Q_n(x)f_n\\
	&=m(x)+m(x)\sum_{\substack{n \in E \\ n\ge 1}}Q_n(x)f_n\lim_{t \to +\infty}\fe_\Phi(t,\lambda_n)=m(x),
	\end{align*}
	concluding the proof.
\end{proof}
\section{Correlation structure of non-local solvable birth-death processes}\label{Sec6}
Let us consider the potential measure $U_\Phi(t)=\E[E_\Phi(t)]$ of the subordinator $\sigma_\Phi(t)$. As a consequence of Corollary \ref{corcov}, we can directly apply \cite[Theorem $2$]{ascione2019semi} to obtain and expression of the covariance of $N_\Phi(t)$ in terms of $\fe_\Phi(t;\lambda_1)$. In particular, with an easy refinement of the proof, by using the distributional derivative of the potential function $U_\Phi(t)$, we can get rid of some hypotheses of the aforementioned Theorem.
\begin{prop}\label{propcov}
	Let $N(t)$ be a solvable birth-death process with state space $E$, invariant measure $\mm$ and family of associated classical orthogonal polynomials $(P_n)_{n \in E}$. Let us denote by $\iota:E \to E$ the identity function and suppose that $\iota=a_0+a_1Q_1$. Then it holds, for any $t \ge s\ge 0$
	\begin{equation*}
	\Cov_m(N_\Phi(t),N_\Phi(s))=a_1^2\left(-\lambda_1\int_0^{s}\fe_\Phi(t-\tau;\lambda_1)dU_\Phi(\tau)-2+2\fe_\Phi(s;\lambda_1)+\fe_\Phi(t;\lambda_1)\right).
	\end{equation*}
\end{prop}
As we expected, since we are composing a stationary process $N(t)$ (if $N(0)$ admits $\mm$ as distribution) with a non-stationary one $E_\Phi(t)$, $N_\Phi(t)$ is not second-order stationary. To introduce the notion of memory for our process $N_\Phi(t)$, we refer to the necessary conditions given in \cite[Lemmas $2.1$ and $2.2$]{beran2016long}, since, for our processes, it is easier to work with the auto-covariance function. In particular we focus on the \textit{long memory with respect to the initial state}.
\begin{defn}
	Given the function $\gamma(n)=\Cov_m(N_\Phi(n),N_\Phi(0))$ for $n \in \N$:
	\begin{itemize}
		\item $N_\Phi(t)$ is said to exhibit long-range dependence if $\gamma(n)\sim \ell(n)n^{-k}$ where $\ell(n)$ is a slowly varying function and $k \in \left(0,1\right)$;
		\item $N_\Phi(t)$ is said to exhibit short-range dependence if $\sum_{n=1}^{+\infty}|\gamma(n)|<+\infty$.
	\end{itemize}
\end{defn}
In the specific case of the initial state, we have the following Proposition.
\begin{prop}\label{covindat}
	Let $N(t)$ be a solvable birth-death process with state space $E$, invariant measure $\mm$ and family of associated classical orthogonal polynomials $(P_n)_{n \in E}$. Let us denote by $\iota:E \to E$ the identity function and suppose that $\iota=a_0+a_1Q_1$. Then, $\lim_{t \to +\infty}\Cov_m(N_\Phi(t),N_\Phi(0))=0$.\\
	Moreover, if $\fe_\Phi(t;\lambda_1) \sim Ct^{-\alpha}$ as $t \to +\infty$ for some $\alpha \in (0,1)$, then $N_\Phi(t)$ is long-range dependent.
\end{prop}
\begin{proof}
	One has
	\begin{equation*}
	\Cov_m(N_\Phi(t),N_\Phi(0))=a_1^2\fe_\Phi(t;\lambda_1),
	\end{equation*}
	where this identity can be achieved by direct calculations, observing that \linebreak ${\Cov_m(N(t),N(0))=a_1^2e^{\lambda_1 t}}$. The second assertion easily follows from last identity. 
\end{proof}
\begin{rmk}
	Let us observe that, actually, since $\Cov_m(N(t),N(0))=a_1^2e^{\lambda_1 t}$, we can argue by direct calculations without using the regularity hypotheses on $\bP(E_\Phi(t)\ge s)$.
\end{rmk}
Concerning the asymptotic behaviour of $\fe_\Phi(t;\lambda_1)$, one can obtain some information by the behaviour of $\Phi$. Indeed one has the following result.
\begin{prop}\label{prop:regvar}
	If $\Phi$ is regularly varying at $0^+$ with order $\alpha \in (0,1]$, then, for any fixed $\lambda>0$, $\fe_\Phi(t;-\lambda)$ is regularly varying at $\infty$ with order $-\alpha$.
\end{prop}
\begin{proof}
	Let us consider $J(t)=\int_0^t \fe_\Phi(s;-\lambda)ds$ and let us observe that the Laplace-Stieltjes transform $\bar{J}(\eta)$ of $J(t)$ is given by
	\begin{equation*}
	\bar{J}(\eta)=\frac{\Phi(\eta)}{\eta(\Phi(\eta)+\lambda)}.
	\end{equation*}
	Since $\lambda$ is positive, one has that $\bar{J}(\eta)$ is regularly varying at $0^+$ if and only if $\frac{\Phi(\eta)}{\eta}$ is (since $\Phi(0^+)=0$ by the fact that $\Phi$ is a driftless Bernstein function, see, for instance, \cite{schilling2012bernstein}). In particular, since $\Phi$ is regularly varying at $0^+$ with order $\alpha$, then $\bar{J}(\eta)$ is regularly varying at $0^+$ with order $\alpha-1$.\\
	By Karamata's Tauberian theorem (see, for instance, \cite{mikosch1999regular} for a compact statement and \cite{bingham1989regular} for the full proof), we know that $J(t)$ is regularly varying at infinity with order $1-\alpha$. Now let us observe that $J'(t)=\fe_\Phi(t,-\lambda)$, that is monotone, hence, by Monotone density theorem (again, see \cite{bingham1989regular,mikosch1999regular}), we have that $\fe_\Phi(t,-\lambda)$ is regularly varying at $\infty$ of order $-\alpha$, concluding the proof.
\end{proof}
In particular we get the following Corollary.
\begin{cor}\label{corlongrange}
	Under the hypotheses of Proposition \ref{covindat}, if $\Phi$ is regularly varying at $0^+$ with order $\alpha \in (0,1)$, then $N_\Phi(t)$ is long-range dependent.
\end{cor}
Let us reconsider the examples given in Section \ref{Sec3} and study the asymptotic behaviour of the covariance.
\begin{itemize}
	\item In the case $\Phi(\lambda)=\lambda^\alpha$ we can actually obtain an explicit formulation of the autocovariance function for $t \ge s >0$:
	\begin{equation*}
	\Cov_m(N_\Phi(t),N_\Phi(s))=a_1^2\left(E_\alpha(\lambda_1 t^\alpha)-\frac{\lambda_1 \alpha t^\alpha}{\Gamma(1+\alpha)}\int_0^{\frac{s}{t}}\frac{E_\alpha(\lambda_1t^\alpha(1-z)^\alpha)}{z^{1-\alpha}}dz\right),
	\end{equation*}
	where the proof of such formula is identical to the one of \cite[Theorem $3.1$]{leonenko2013correlation}. Thus we can use \cite[Remark $3.3$]{leonenko2013correlation} to obtain directly long-range dependence of the process.
	\item As $\Phi(\lambda)=(\lambda+\theta)^\alpha-\theta^\alpha$, we have that $\Phi(\lambda)$ is actually regularly varying at $0^+$ of order $1$. Indeed we have $\Phi(\lambda)/\lambda \to \alpha \theta^{\alpha+1}$ as $\lambda \to 0^+$. In particular this means that the function $\overline{J}(\eta)$ defined in Proposition \ref{prop:regvar} is such that $\overline{J}(0+)=\frac{\alpha \theta^{\alpha+1}}{\lambda}$. By Karamata's Tauberian theorem, we still have $\lim_{t \to +\infty}J(t)=\frac{\alpha \theta^{\alpha+1}}{\lambda}$. This means in particular that $\fe_\Phi(s;-\lambda)$ is integrable in $(0,+\infty)$ and thus $\sum_{n=1}^{+\infty}\fe_\Phi(n;-\lambda)<+\infty$. From this observation we obtain that in such case $N_\Phi(t)$ is short-range dependent.
	\item If we consider $\Phi(\lambda)=\log(1+\lambda^\alpha)$ for $\alpha \in (0,1)$, then $\Phi(\lambda)$ is regularly varying at $0^+$ of order $\alpha$, since $\lim_{\lambda \to 0^+}\frac{\log(1+\lambda^\alpha)}{\lambda^\alpha}=1$. Thus, in particular, $N_\Phi$ is long-range dependent by Corollary \ref{corlongrange}.
	\item Finally, if $\Phi(\lambda)=\log(1+\lambda)$, we have $\lim_{\lambda \to 0^+}\frac{\log(1+\lambda)}{\lambda}=1$ and then $\overline{J}(0+)=\frac{1}{\lambda}$. This means again that $\fe_\Phi(t;-\lambda)$ is integrable in $(0,+\infty)$ and then $N_\Phi$ is short-range dependent.
\end{itemize}
\bibliographystyle{plain}
\bibliography{bib}
\end{document}